\documentclass[letter,english]{article}

\usepackage[utf8]{inputenc}
\usepackage[T1]{fontenc}
\usepackage{babel}
\usepackage[babel]{csquotes}

\usepackage{amsmath}
\usepackage{amsfonts}
\usepackage{amsthm}
\usepackage{bbm}
\usepackage{amssymb}
\usepackage{IEEEtrantools}
\usepackage{url}

\newtheorem{theorem}{Theorem}[section]
\newtheorem{proposition}[theorem]{Proposition}
\newtheorem{lemma}[theorem]{Lemma}
\newtheorem{corollary}[theorem]{Corollary}
\theoremstyle{definition}
\newtheorem{definition}[theorem]{Definition}
\newtheorem{example}{Example}
\newtheorem{remark}{Remark}

\DeclareMathOperator{\Mat}{Mat}
\DeclareMathOperator{\Ad}{Ad}
\DeclareMathOperator{\ad}{ad}
\DeclareMathOperator{\Stab}{Stab}

\DeclareMathOperator{\End}{End}
\DeclareMathOperator{\id}{id}

\DeclareMathOperator{\card}{card}
\DeclareMathOperator{\Span}{Span}

\DeclareMathOperator{\Lie}{Lie}

\newcommand{\C}{\mathbb{C}}
\newcommand{\R}{\mathbb{R}}

\newcommand{\g}{\mathfrak{g}}
\newcommand{\h}{\mathfrak{h}}
\newcommand{\G}{\mathcal{G}}

\title{A product theorem in simple Lie groups}
\author{Nicolas de Saxcé
\thanks{The author is supported by ERC AdG Grant 267259}
}

\begin{document}

\maketitle

\begin{abstract}
We prove a discretized Product Theorem for general simple Lie groups, in the spirit of Bourgain's Discretized Sum-Product Theorem. 
\end{abstract}

\section{Introduction}

The goal of this paper is to prove a discretized Product Theorem for simple Lie groups. The theorem is a growth statement in the spirit of Bourgain's \enquote{discretized sum-product} \cite{bourgainringconjecture, bourgainprojection}, but in the context of simple Lie groups.

If $A$ is a subset of a compact metric space, for $\delta>0$, we denote by $N(A,\delta)$ the minimal cardinality of a cover of $A$ by balls of radius $\delta$.
The theorem we prove is the following.

\begin{theorem}[Product Theorem]\label{sigmai}
Let $G$ be a simple real Lie group of dimension $d$. There exists a neighborhood $U$ of the identity in $G$ such that the following holds.\\
Given $\sigma\in (0,d)$, there exists $\epsilon=\epsilon(\sigma)>0$ such that, for $\delta>0$ sufficiently small, if $A\subset U$ is a set satisfying
\begin{enumerate}
\item $N(A,\delta)\leq\delta^{-\sigma-\epsilon}$
\item $\forall \rho\geq\delta$, $N(A,\rho) \geq \delta^{\epsilon}\rho^{-\sigma}$
\item for any closed connected subgroup $H\subset G$, there exists $a\in A$ with $d(a,H)>\delta^\epsilon$,
\end{enumerate}
then
$$N(AAA,\delta)>\delta^{-\epsilon} N(A,\delta).$$
\end{theorem}

For the group $SU(2)$, Theorem~\ref{sigmai} is due to Bourgain and Gamburd \cite[Proposition~6]{bourgaingamburdsu2}, and for the group $SL(2,\mathbb{R})$, to Bourgain and Yehudayoff \cite[Theorem~4.4]{bourgainyehudayoff}. In both cases, the proof is based on Helfgott's argument in \cite{helfgottsl2}, using trace to show expansion. Our approach in the present paper is slightly different; the proof is based on the strategy developed by Bourgain and Gamburd in \cite{bourgaingamburdsud}, taking advantage of the scale invariance property of the set $A$. We also make use of some ideas of Breuillard, Green and Tao \cite{breuillardgreentao} (see also Pyber and Szabó \cite{pyberszabo}) for the proof of some Larsen-Pink type inequalities.

\bigskip

Interest in discretized results of the type of Theorem~\ref{sigmai} started with the work of Katz and Tao \cite{katztao} and later of Bourgain \cite{bourgainringconjecture,bourgainprojection} on the Erd\H{o}s-Volkmann Ring Conjecture. Since then, those discretized results have found many applications, among which the work of Bourgain, Furman, Lindenstrauss and Mozes on quantitative equidistribution of orbits of semigroups on the torus \cite{bflm} and that of Bourgain and Gamburd \cite{bourgaingamburdsu2,bourgaingamburdsud} on the spectral gap property for finitely generated subgroups of $SU(d)$. In fact, our product theorem can be used to prove the spectral gap property for subgroups generated by algebraic elements in an arbitrary compact simple Lie group \cite{benoistsaxcesg}.

\paragraph{}

The plan of the paper is as follows. Section~\ref{nonconcentration} is devoted to the proof of some Larsen-Pink type inequalities for approximate subgroups of $G$. The proof of the Product Theorem~\ref{sigma} is given in Section~\ref{expansion}.

\paragraph{}

For us, a simple Lie group will be a real Lie group whose Lie algebra is simple. We will also make use of some classical notation:
\begin{itemize}
\item[-] The Landau notation: $O(\epsilon)$ stands for a quantity bounded in absolute value by $C\epsilon$, for some constant $C$ (generally depending on the ambient group $G$).
\item[-] The Vinogradov notation: we write $x\ll y$ if $x\leq Cy$ for some constant $C$ (again, possibly depending on the ambient group). We will also write $x\simeq y$ if $x\ll y$ and $x\gg y$.
\end{itemize}

\paragraph{Acknowledgements}
I am very grateful to Yves Benoist for his good advice and for his detailed comments on an earlier version of the paper. I also thank Mike Hochman, Elon Lindenstrauss and Péter Varjú for useful discussions, and Emmanuel Breuillard for introducing me to this topic, during my doctoral thesis under his supervision.

\section{Larsen-Pink type inequalities}\label{nonconcentration}

\subsection{Escaping from subvarieties}

Our goal is to show here that if a subset $A$ of a simple Lie group $G$ is away from any closed subgroup -- in a quantitative sense given below -- then for any algebraic subvariety $V$ of $G$, one can obtain in a product set of $A$ an element that is away from $V$. The idea of this \enquote{escape from subvarieties} originates in the work of Eskin-Mozes-Oh \cite{eskinmozesoh}. The main difference here is that we will need a lower bound for the distance to the subvariety from which we want to escape.\\

We start by a preliminary proposition that describes the shape of maximal connected subgroups of a simple Lie group in a neighborhood of the identity.

\begin{proposition}\label{maximal}
Let $G$ be a simple Lie group. There exists a neighborhood $O$ of $0$ in the Lie algebra $\g$ of $G$ such that the exponential map induces a diffeomorphism from $O$ to its image $U$ in $G$ and moreover, whenever $H$ is a maximal proper closed connected subgroup,
$$\exp: O\cap \h \rightarrow U\cap H \mbox{is a diffeomorphism}.$$
\end{proposition}
\begin{proof}
Choosing a neighborhood $O$ such that the exponential map induces a diffeomorphism from $O$ to its image $U$ in $G$, we want to ensure that for any maximal subgroup $H$, whenever $x\in U\cap H$, one has $X:=\log x \in\h$. 
As $H$ is maximal, it must be equal to the identity component of the normalizer of its Lie algebra in the adjoint representation, so what we have to check is that $(\ad X)\h \subset \h$.  The following lemma exactly says that this can always be ensured by choosing $O$ small enough.
\end{proof}

\begin{lemma}
Let $G$ be a Lie group. There exists a neighborhood $O$ of the identity in $\g$ such that for any $X\in O$ and any linear subspace $\h$ in $\g$,
$$(\Ad e^X)\h \subset \h \Longleftrightarrow (\ad X)\h \subset \h.$$
\end{lemma}
\begin{proof}
Indeed, take a neighborhood $O$ of zero in $\g$ such that for all $X$ in $O$, one has
$$\ad X = \sum_{n\geq 1} \frac{(1-e^{\ad X})^n}{n},$$
and suppose $(\Ad e^X)\h \subset \h$. Let $Y\in \h$. Using the identity $\Ad e^X = e^{\ad X}$, we see that for all $n$, $(1-e^{\ad X})^nY \in \h$ and therefore, as $\h$ is closed,
$$(\ad X)Y = \sum_{n\geq 1} \frac{(1-e^{\ad X})^nY}{n}\in \h.$$
\end{proof}

In the case the simple Lie group $G$ has trivial center, it is equal to the (connected component of the) group of real points of a simple linear algebraic group, and Proposition~\ref{maximal} yields the following.

\begin{lemma}\label{neutre}
Let $G$ be a simple Lie group with trivial center. There exists a neighborhood $U$ of the identity in $G$ such that for any $g\in U$, for any maximal proper algebraic subgroup $H$,
$$d(g,H)=d(g,H^0),$$
where $H^0$ is the identity component of $H$.
\end{lemma}
\begin{proof}
It suffices to show that there is a neighborhood $U$ of the identity in $G$ such that for any maximal proper algebraic subgroup $H$,
\begin{equation}\label{uh}
U\cap H = U\cap H^0.
\end{equation}
If the Lie algebra $\h$ of $H$ is nonzero, then by maximality, $H$ is equal to the normaliser of $\h$, so the previous lemma shows that we can find $U$ such that (\ref{uh}) holds for any positive dimensional maximal $H$. To deal with finite maximal subgroups, we use Jordan's Theorem: there is a constant $C$ depending on $G$ only such that if $H$ is a finite subgroup, there exists a torus $T$ in $G$ such that $H$ is included in the normalizer of $T$ and $[H:H\cap T]\leq C$. If $H$ is maximal, we must have $T=\{1\}$ and therefore $|H|\leq C$. Taking $U$ to be of the form $\exp B(0,\frac{r}{C})$ where $r$ is such that the exponential is one-one on $B(0,r)$, we indeed find $H\cap U=\{1\}$.
\end{proof}

In order to satisfy the desired \enquote{escape-from-subvariety} property, a set $A$ should not be too close to closed subgroups of $G$. That is what is quantified in the following definition.

\begin{definition}
Let $\frac{1}{2}>\rho>0$ be a parameter. We say that a subset $A$ of a connected Lie group $G$ is \emph{$\rho$-away from subgroups} if for any proper closed connected subgroup $H$, there exists an element $a$ in $A$ such that $d(a,H)\geq\rho$.
\end{definition}

We start by an elementary observation.

\begin{lemma}\label{dtuple}
Let $G$ be a simple Lie group. There exists a neighborhood $U$ of the identity in $G$ and a constant $C=C(G)\geq 0$ such that if $A\subset U$ is $\rho$-away from subgroups, then $A$ contains a subset of cardinality at most $d$ that is $\rho^C$-away from subgroups.
\end{lemma}
\begin{proof}
Let $U$ be an exponential neighborhood of the identity, and denote as before $O=\log U$.
By Proposition~\ref{maximal}, we may assume that for any maximal proper closed connected subgroup $H$ of $G$, the intersection $H\cap U$ is equal to $\exp(\mathfrak{h}\cap O)$, where $\mathfrak{h}$ is the Lie algebra of $H$.
Suppose $A$ is included in $U$ and is $\rho$-away from subgroups. We define inductively the elements $a_i$ of the desired finite subset. First, choose $a_1$ in $A$ such that $d(a_1,0)\geq \rho$. Now assume the $a_i$'s, $1\leq i\leq k$, are defined. If $\{a_i\}_{1\leq i\leq k}$ is $\rho^C$-away from subgroups, we are done, and we do not need to define $a_{k+1}$. Otherwise, there exists a maximal closed subgroup $H_k$ such that for each $i\leq k$,
$$d(a_i,H_k)\leq\rho^C.$$ 
Using the fact that $A$ is $\rho$-away from subgroups, we pick in $A$ an element $a_{k+1}$ such that $d(a_{k+1},H_k)\geq\rho$. We just have to check that this procedure stops for some $k\leq d$. For that, we write $X_i=\log a_i$ and $\mathfrak{h}_k=\Lie H_k$, so $d(X_i,\mathfrak{h}_k)\leq\rho^C$. We will show that at each stage, the family $(X_i)_{1\leq i\leq k}$ is linearly independent (this forces in particular $k\leq d$).\\
Let $V_k=\Span(X_i)_{1\leq i\leq k}$. By induction on $k\geq 0$, we check that, $d(X_{k+1},V_k)\geq \rho/2$. Assume the result holds for all $j$ in $\{0,\dots,k\}$. In particular, any element $X$ of $O\cap V_k$ can be written $X=\sum_{i\leq k} \lambda_i X_i$ with $|\lambda_i|\leq \rho^{-C_0}$, for some constant $C_0=C_0(d)$, see Lemma~\ref{bilipschitz}. Therefore, any element of $O\cap V_k$ is $d\rho^{C-C_0}$-close to $\mathfrak{h}_k$, and thus away from $X_{k+1}$ by at least $\frac{\rho}{2}$, provided we have chosen $C>C_0+2$.
\end{proof}

In order to prove the escape property, the strategy will be to linearize the variety in some finite dimensional linear representation of $G$. But first, we show that given a representation of $G$ on a finite dimensional space $V$, if $A$ is away from subgroups, then no linear subspace of $V$ can be fixed under all elements of $A$.

\begin{definition}
Let $V$ be a finite-dimensional Hilbert space. Given $W$ and $W'$ two subspaces, we define the distance from $W$ to $W'$ by
$$d(W,W')=\max\{ d(u,W')\,;\, u\in W\ \mbox{and}\ \|u\|=1\}.$$
\end{definition}

Note that $d(W,W')=0$ if and only if $W$ is contained in $W'$. In the case $W$ and $W'$ have the same dimension $l$, we have $d(W,W')=d(W',W)$ and therefore $d$ is a distance on the Grassmannian variety of subspaces of dimension $\ell$.

\begin{proposition}\label{almoststab}
Let $G$ be a simple Lie group with trivial center and $V$ be a finite-dimensional complex representation of $G$. There exists a neighborhood $U$ of the identity in $G$ such that the following holds.\\
Given $c>0$, there exist constants $C\geq 0$ and $\rho_0>0$ depending only on $V$ and $c$, such that the following holds for any $\rho\in(0,\rho_0)$.\\
Suppose $A\subset U$ is $\rho$-away from subgroups. If $W$ is a subspace of $V$ such that, for some $x\in U$, $d(x\cdot W,W)\geq c$ then there exists an element $a$ in $A$ such that
$$d(a\cdot W,W)\geq \rho^C.$$
\end{proposition}
\begin{proof}
As $V$ is finite-dimensional, we may assume that the dimension of the subspace $W$ is fixed, equal to $\ell$. The action of $G$ on $V$ is algebraic and hence, so is the induced action of $G$ on the Grassmannian $\G_\ell$ of $\ell$-dimensional subspaces of $V$. Let $U$ be a compact neighborhood of the identity in $G$.\\
The map
$$
\begin{array}{cccc}
f: & G^d\times \G_\ell & \rightarrow & \R\\
& (\bar{g},\xi) & \mapsto & \sum_{i=1}^d d(g_i\cdot \xi,\xi)^2
\end{array}
$$
is real-analytic, so that we may apply the \L ojasiewicz inequality \cite[Théorème~2, page 62]{lojasiewicz} (in some local analytic charts),
and get that for some constant $C\geq 0$, for all $(\bar{g},W)\in U^d\times\G_\ell$,
$$\sum_{i=1}^d d(g_i\cdot W, W)^2 \geq \frac{d( (\bar{g},W), Z)^C}{C},$$
where $Z$ is the zero set of $f$:
$$Z=\{(\bar{g},W)\in G^d\times\G_\ell\,|\,\forall i,\  g_i\cdot W=W\}.$$
Now choose $U$ as in Lemma~\ref{neutre}. We claim that if for any proper closed subgroup $H$, the $d$-tuple $\bar{g}$ has a coordinate whose distance to $H$ is bounded below by $\rho$, and if for some $x$ in $U$, $d(x\cdot W,W)\geq c$, then $d( (\bar{g},W),Z)\geq \rho$. Indeed, assume by contrapositive that $d( (\bar{g},W),Z)\leq \rho$. Then there exists $(\bar{g}_0,W_0)$ such that
$$d(\bar{g},\bar{g}_0)\leq \rho, \quad d(W,W_0)\leq \rho,\quad \mbox{and}\quad \forall i,\, g_{0,i}\cdot W_0=W_0.$$
This implies in particular that, for each $i$, $d(g_i,\Stab W_0)\leq \rho$. If $\Stab W_0\neq G$, we can choose a proper maximal algebraic subgroup $H$ containing $\Stab W_0$, and then have, for each $i$,
$$d(g_i,H^0)=d(g_i,H)\leq d(g_i,\Stab W_0)\leq \rho.$$
So we just have to show that $\Stab W_0\neq G$. For this, recall that for some $x\in U$, we have
$$d(x\cdot W,W)\geq c.$$
As $U$ is compact, there is a constant $C_0$ such that all elements of $U$ are $C_0$-Lipschitz, as transformations of $\G_\ell$; in particular,
$$d(x\cdot W_0,W_0)\geq c - 2C_0\rho>0,$$
provided $\rho$ is small enough (depending on $c$), so that $\Stab W_0\neq G$.
This proves the proposition in the case $A$ has finite cardinality at most $d$.\\
By Lemma~\ref{dtuple}, the general case follows from this.
\end{proof}

From the previous lemma, we may now obtain by induction the following quantitative escape property. 

\begin{proposition}[Escape from subvarieties]\label{escape}
Let $G$ be a simple Lie group with trivial center and $V$ be a finite dimensional complex representation of $G$. Fix a neighborhood $U$ of the identity in $G$ for which Proposition~\ref{almoststab} holds.\\
Given $c>0$ there exist constants $C\geq 0$ and $\rho_0>0$ depending only on $V$ and $c$ such that the following holds for any $\rho\in(0,\rho_0)$.\\
Assume that $A\subset U$ is $\rho$-away from subgroups, and that $1\in A$.
Let $v$ be a unit vector in $V$ and $W<V$ a linear subspace of dimension $\ell$ such that for some $x\in U$, $d(x\cdot v,W)\geq c$.\\
Then there exists an element $a\in A^\ell$ such that $d(a\cdot v,W)\geq \rho^C$.
\end{proposition}

If $X$ is a subspace of a metric space and $\rho$ any positive number, $X^{(\rho)}$ denotes the $\rho$-neighborhood of $X$, i.e. the set of points whose distance to $X$ is less than $\rho$. The proof of Proposition~\ref{escape} will use the following simple observation.

\begin{lemma}\label{anglealpha}
Let $V$ be a finite dimensional Hilbert space. Let $W_1$ and $W_2$ be two different subspaces of $V$ of the same dimension, and denote $\alpha=d(W_1,W_2)$. Then there exists $W_0$ a proper subspace of $W_1$ such that for all $r\in (0,1)$,
$$W_1^{(r)}\cap W_2^{(r)} \subset W_0^{(\frac{3r}{\alpha})}.$$
\end{lemma}
\begin{proof}
Let $u$ be a unit vector in $W_1$ such that $d(u,W_2)=\alpha$, and define
$$f=\frac{p_{W_2^\perp}(u)}{\|p_{W_2^\perp}(u)\|},$$
where $p_{W_2^\perp}$ is the orthogonal projection onto $W_2^\perp$. From $\|f\|=1$ and $f^\perp \supset W_2$, on gets, for any $r$,
$$W_2^{(r)} \subset \{v\,|\, |(f,v)|\leq r\}.$$
On the other hand, $|(f,u)|=\alpha$, so that, viewing $f$ as a linear form, we have $\|f|_{W_1}\|\geq\alpha$. We let $W_0=\ker f|_{W_1}=W_1\cap \ker f$. If $v$ is in $W_1\cap W_2^{(r)}$, we have
$$d(v,W_0) = \frac{|(f,v)|}{\|f|_{W_1}\|} \leq \frac{r}{\alpha}.$$
This shows that $W_1\cap W_2^{(r)} \subset W_0^{(\frac{r}{\alpha})}$. Noting that $W_1^{(r)}\cap W_2^{(r)}$ is included in a neighborhood of size $r$ of $W_1\cap W_2^{(2r)}$, this proves the lemma.
\end{proof}

\begin{proof}[Proof of Proposition~\ref{escape}]
We prove the proposition by induction on the dimension $\ell$ of $W$.\\
\underline{$\ell=0$}\\
The result is clear, since $1\in A$ and $d(v,\{0\})=1\geq\rho$.\\
\underline{$\ell\rightarrow \ell+1$}\\
Suppose we have found a constant $C_\ell$ depending only on $U$, $V$ and $c$ such that the proposition holds for any subspace $W$ of dimension less than or equal to $\ell$.\\
Let $L\geq 1$ be a constant such that all elements of $U$ are $L$-Lipschitz, as diffeomorphisms of $V$.
As $1\in A$, we may assume without loss of generality that $d(v,W)\leq \frac{c}{2L}$, so that choosing $w\in W$ such that $d(v,w)\leq\frac{c}{2L}$, we find that for some $x\in U$,
$$d(x\cdot w,W)\geq d(x\cdot v,W)-Ld(v,w)\geq \frac{c}{2}$$
which implies
$$d(x\cdot W,W)\geq \frac{c}{2}.$$
Therefore, from Proposition~\ref{almoststab}, we may find $a\in A$ such that $d(a^{-1}\cdot W,W)\geq \rho^{C_0}.$ By Lemma~\ref{anglealpha}, this implies
 that for some proper subspace $W_0<W$, for all $r>0$, the intersection $W^{(r)}\cap a^{-1}\cdot  W^{(r)}$ lies in $W_0^{(3r\rho^{-C_0})}$.\\
We will prove the proposition with  constant $C=C_0+C_\ell+1$.\\
Of course, if $d(v,W)\geq \rho^{C}$, there is nothing to prove, so we assume $d(v,W)\leq \rho^{C}$ and choose $w\in W$ such that
$$d(v,w)\leq\rho^{C}.$$
From the induction hypothesis applied to $W_0$, there exists an $a_\ell\in A^\ell$ such that
\begin{equation}\label{close}
d(a_\ell\cdot w,W_0)\geq \rho^{C_\ell}.
\end{equation}
If $d(a_\ell\cdot v, W)\geq \rho^{C}$ we are done.\\
Otherwise, we must have $d(a_\ell\cdot w,W)\leq (L^\ell+1)\rho^C$. Suppose for a contradiction that $d(a\cdot (a_\ell w),W)\leq (L^{\ell+1}+1)\rho^{C}$; then $d(a_\ell w,a^{-1}W)\leq L(L^{\ell+1}+1)\rho^C$ and so $a_\ell w \in W^{((L^\ell+1)\rho^C)}\cap a^{-1}W^{(L(L^{\ell+1}+1)\rho^C)}$, which implies, by definition of $W_0$, for $\rho>0$ small enough,
$$d(a_\ell w, W_0) \leq 3L(L^{\ell+1}+1)\rho^{C-C_0} < \rho^{C_\ell},$$
contradicting (\ref{close}).
Thus, we find,
$$d(aa_\ell\cdot v,W)\geq d(aa_\ell\cdot w,W)-L^{\ell+1}d(v,w)\geq (L^{\ell+1}+1)\rho^{C}-L^{\ell+1}\rho^C \geq \rho^C.$$
\end{proof}

\begin{remark}\label{irred}
One can check that the map $\varphi:(v,W)\mapsto \max_{g\in U} d(gv,W)$ is continuous,
 and if $V$ is an irreducible representation, it is also everywhere positive. Using compacity of the unit sphere in $V$ and of the Grassmannian variety of hyperplanes, this shows that there exists a small constant $c>0$ depending only on $U$ and $V$ such that for any unit vector $v\in V$, and any subspace $W<V$, there exists $x\in U$ such that $d(x\cdot v,W)\geq c$. So, in the case where $V$ is irreducible, the proposition holds without any assumption on $v$ and $W$.
\end{remark}

\subsection{Larsen-Pink type estimates}

The purpose of this section is to derive some metric analogs of the Larsen-Pink type inequalities, as developped by Breuillard, Green and Tao \cite{breuillardgreentao}, and by Pyber and Szab\'o \cite{pyberszabo}. Because we needed to take into account the metric of the ambient space, it seemed more natural to work with differential submanifolds, rather than algebraic subvarieties. Thus, we first define a notion of complexity in this setting, and then prove the needed Larsen-Pink estimates.

\paragraph{}

As before, the letter $C$ denotes a large positive constant, whose value may increase from one line to the other, but depending only on the ambient dimension $d$ or on the ambient group $G$. We will say that a map $f$ between two metric spaces $E$ and $F$ is $K$-Lipschitz if it satisfies, for all $x$ and $y$ in $E$, $d(f(x),f(y))\leq K d(x,y)$. If $f$ is bijective, we say that it is $K$-bi-Lipschitz if both $f$ and $f^{-1}$ are $K$-Lipschitz.

\subsubsection{Complexity of submanifolds of $\R^d$}

We start by defining complexity for diffeomorphisms defined on an open ball of the Euclidean space $\R^d$.

\begin{definition}
Let $\frac{1}{2}>\rho>0$ be a parameter.
A \emph{diffeomorphism of complexity} $\rho^{-1}$ is a map $f$ from $B_\rho:=B_{\R^d}(0,\rho)$ to $\R^d$ satisfying the following properties:
\begin{itemize}
\item $f(0)=0$
\item $f$ is a diffeomorphism of $B_\rho$ onto its image
\item $f'(0):\R^d\rightarrow\R^d$ is $\rho^{-1}$-bi-Lipschitz
\item the differential of $f$, $f':B_\rho\rightarrow \End\R^d$ is $\rho^{-1}$-Lipschitz.
\end{itemize}
\end{definition}

The first thing we want to check is that inverses and compositions of diffeomorphisms of bounded complexity remain of bounded complexity. This will be a straightforward application of the following quantitative version of the Inverse Function Theorem.

\begin{theorem}[Quantitative Inverse Function Theorem]\label{quantitative}
There exists an absolute constant $C$ ($C=3$) such that the following holds for any $\frac{1}{2}>\rho>0$. Let $f$ be a $C^1$ map from $\mathbb{R}^d$ to $\mathbb{R}^d$ satisfying:
\begin{enumerate}
\item The map $f'(0):\mathbb{R}^d\rightarrow\mathbb{R}^d$ is $\rho^{-1}$-bi-Lipschitz,
\item The map $f':\mathbb{R}^d\rightarrow\End\mathbb{R}^d$ is $\rho^{-1}$-Lipschitz,
\end{enumerate}
then, $f$ induces a $\rho^{-C}$-bi-Lipschitz $C^1$-diffeomorphism from $B_{\rho^C}$ onto its image.
\end{theorem}

The proof is the same as for the usual Local Inverse Theorem, but one has to keep track of the constants. The key lemma is the following.
\begin{lemma}
Let $\Omega$ be an open subset of $\mathbb{R}^d$, $\varphi:\Omega\rightarrow\mathbb{R}^d$ a $k$-Lipschitz map, with $k<1$ and $f=j+\varphi$, where $j$ is the canonical injection from $\Omega$ to $\mathbb{R}^d$. Then $f$ is a $\frac{1}{1-k}$-bi-Lipschitz homeomorphism from $\Omega$ onto $f(\Omega)$.
\end{lemma}
\begin{proof}
Proof is a simple application of Picard's Fixed Point Theorem, we leave it to the reader.
\end{proof}
\begin{proof}[Proof of Theorem~\ref{quantitative}]
First assume $f'(0)=id$. As $f'$ is $\rho^{-1}$-Lipschitz, on a ball of radius $\frac{\rho}{2}$, the function $\varphi=f-id$ is $\frac{1}{2}$-Lipschitz, so $f$ induces a $2$-bi-Lipschitz $C^1$-diffeomorphism on that ball, and we are done.\\
The general case follows from considering $\tilde{f}=(f'(0)^{-1})\circ f$.
\end{proof}

\begin{proposition}\label{composition}
There exists an absolute constant $C$ ($C=5$) such that for any $\frac{1}{2}>\rho>0$, the following holds. Let $f$ and $g$ be diffeomorphisms of complexity $\rho^{-1}$. Then $f\circ g$ and $f^{-1}$ are diffeomorphisms of complexity $\rho^{-C}$.
\end{proposition}
\begin{proof}
From the quantitative Local Inverse Theorem, one sees that there exists an absolute constant $C$ such that if $g$ is a diffeomorphism of complexity $\rho^{-1}$, then $g$ is $\rho^{-C}$-bi-Lipschitz on a ball of size $\rho^C$.
In particular, the image of a ball of size $\rho^{C}$ under $g$ is included in $B_\rho$, and therefore, $f\circ g$ is well-defined on $B_{\rho^{C}}$. Of course, $f\circ g(0)=0$, $f\circ g$ is a diffeomorphism of $B_{\rho^{C}}$ onto its image, and $(f\circ g)'(0)=f'(0)\circ g'(0)$ is $\rho^{-C}$-bi-Lipschitz. Finally, from the properties of $f'$, $g$, and $g'$, one readily checks that $(f\circ g)'=(f'\circ g)\cdot g'$ is $\rho^C$-Lipschitz on a ball of radius $\rho^C$. This shows that $f\circ g$ is of complexity $\rho^{-C}$.\\
The proof of the statement concerning $f^{-1}$ is similar, we leave it to the reader.
\end{proof}

\begin{remark}
To be more accurate, we should say that the restrictions of $f\circ g$ and $f^{-1}$ to the ball $B_{\rho^C}$ are of complexity $\rho^{-C}$. However, for brevity, we will continue with this abuse of language.
\end{remark}

\begin{definition}
A \emph{submanifold chunk of dimension $m$ of complexity} $\rho^{-1}$ in $\R^d$ is the image of $B_\rho\cap\R^m$ under a diffeomorphism of complexity $\rho^{-1}$. Note that by definition, a submanifold chunk always contains $0$.
\end{definition}

\begin{remark}\label{tangent}
If $M$ is a submanifold chunk of complexity $\rho$ one may always find for $M$ a defining diffeomorphism $f$ of complexity $\rho^{-C}$ such that $f'(0)=\id$ and for all $x\in B_{\rho^C}$, $f(x)-x\in T_0M^\perp$, where $T_0M$ is the tangent space to $M$ at $0$.
\end{remark}

\begin{lemma}
There exists an absolute constant $C$ ($C=5$) such that the image of a submanifold chunk of complexity $\rho^{-1}$ under an application of complexity $\rho^{-1}$ is a submanifold chunk of complexity $\rho^{-C}$.
\end{lemma}
\begin{proof}
This follows from the definition of a chunk of complexity $\rho^{-1}$, together with the fact that a composition of diffeomorphisms of complexity $\rho^{-1}$ has complexity $\rho^{-C}$, for some absolute constant $C$ (Proposition~\ref{composition}).
\end{proof}

Our goal now is to check that the intersection of two transverse submanifold chunks of bounded complexity is again of bounded complexity. We start by some elementary observations on angles between linear subspaces of $\R^d$.

\begin{lemma}\label{bilipschitz}
For any positive integer $d$, there exists a constant $C$ ($C=2d$) such that the following holds.\\
Let $\frac{1}{2}\geq\rho>0$ and $(u_i)$ a family of unit vectors in a Euclidean space $E$ of dimension $d$. Suppose that for each $i\in\{1,\dots,d\}$,
$$d(u_i,\bigoplus_{j=1}^{i-1}\R u_j)\geq\rho.$$
Then, the map
$$
\theta:
\begin{array}{ccc}
\mathbb{R}^d&\rightarrow& E\\
(t_i)&\mapsto&\sum t_iu_i
\end{array}
$$
is $\rho^{-C}$-bi-Lipschitz ($\mathbb{R}^d$ Euclidean).
\end{lemma}
\begin{proof}
First, for any $t=(t_1,\dots,t_d)\in\R^d$,
$$\left\|\sum t_iu_i\right\| \leq \sum|t_i| \leq \sqrt{d}\|t\| \leq \rho^{-d}\|t\|,$$
so $\theta$ is $\rho^{-d}$-Lipschitz.\\
On the other hand, denote $v_j=\sum_{i=1}^j t_iu_i$ and write 
$$\|v_d\|^2=\|v_{d-1}\|^2+t_d^2+2t_d(v_{d-1},u_d).$$
From the assumption on the $u_i$'s, the angle $\alpha$ between $\frac{v_{d-1}}{\|v_{d-1}\|}$ and $u_d$ satisfies $|\cos\alpha|\leq 1-\frac{\rho^2}{2}$ and therefore,
\begin{align*}
\|v_d\|^2 & \geq \|v_{d-1}\|^2+t_d^2-2|t_d|\|v_{d-1}\|(1-\frac{\rho^2}{2})\\
& \geq \frac{\rho^2}{2}(\|v_{d-1}\|^2+t_d^2).
\end{align*}
Using the same argument, we can also bound below $\|v_{d-1}\|$, $\|v_{d-2}\|$, ...etc. At the end, we get
$$\|v_d\|^2 \geq \frac{\rho^{2d}}{2^d} \sum t_i^2 \geq \rho^{4d} \sum t_i^2,$$
i.e. $\theta^{-1}$ is $\rho^{-2d}$-Lipschitz.
\end{proof}

\begin{definition}
Let $E$ be a Euclidean space of dimension $d$. Suppose $F_0$ is a hyperplane in $E$ and $F_1$ is a proper linear subspace of $E$. If $F_0^\perp\subset F_1$, we say that $F_0$ and $F_1$ \emph{form a square angle}.
\end{definition}

\begin{lemma}\label{bilipschitz2}
Let $E$ be a Euclidean space of dimension $d$. There exists a constant $C\geq 0$ ($C=8d$) such that the following holds.\\
Let $\frac{1}{2}\geq\rho>0$ be a parameter. Suppose $F_0$ is a hyperplane in $E$ and $F_1$ is a proper linear subspace of $E$ such that
$$d(F_1,F_0) \geq \rho,$$
then there exists a $\rho^{-C}$-bi-Lipschitz linear automorphism $\theta$ of $E$ fixing $F_0$ and such that $\theta(F_1)$ and $F_0$ form a square angle.
\end{lemma}
\begin{proof}
Start with an orthonormal basis $(u_i)_{1\leq i\leq d-1}$ for $F_0$, and let $u_d$ be a unit vector in $F_0^\perp$.
As $d(F_1,F_0) \geq \rho$, there is a unit vector $v$ in $F_1$ such that $d(v,F_0)\geq \rho$.
The basis $(u_i)_{1\leq i\leq d-1}\cup \{v\}$ satisfies the assumptions of Lemma~\ref{bilipschitz}, and therefore, Lemma~\ref{bilipschitz} shows that the linear map $\theta$ fixing $F_0$ and mapping $v$ to $u_d$ is $\rho^{-C}$-bi-Lipschitz, for some $C$ depending on $d$ only, so we are done.
\end{proof}

We will now use the above lemma to study the intersection of two submanifold chunks of bounded complexity, one of them having codimension 1.

\begin{lemma}[Intersection of transverse chunks]\label{linearchunk}
For each positive integer $d$, there exists a constant $C$ ($C=250d$) depending only on $d$ such that the following holds for any $\rho\in(0,\frac{1}{2})$.\\
Let $M$ and $N$ be two submanifold chunks of complexity $\rho^{-1}$, and satisfying the following:
\begin{itemize}
\item $\dim M = d-1$
\item $d(T_0N,T_0M) \geq \rho$.
\end{itemize}
Then, $M\cap N$ is a submanifold chunk of complexity $\rho^{-C}$. More precisely, there exists a diffeomorphism of complexity $\rho^C$ that sends $M$ and $N$ onto two linear subspaces $F_M$ and $F_N$ forming a square angle. 
\end{lemma}
\begin{proof}
First, using Lemma~\ref{bilipschitz2}, we may compose by a $\rho^{-C}$-bi-Lipschitz linear transformation, and reduce to the case when $T_0M$ and $T_0N$ form a square angle.\\
Then, if $f$ is the diffeomorphism defining $M$, we may compose by $f^{-1}$, and thus assume without loss of generality that $M=F_M$ is a linear subspace.\\
Finally, from the Remark~\ref{tangent} above, we may assume that $N$ is the image of $T_0N$ under a diffeomorphism $g$ of complexity $\rho^{-C}$ satisfying, for all $x\in B_{\rho^C}$, $x-g(x)\in (T_0N)^\perp$. In particular, as $T_0N^\perp\subset T_0M$, $M=F_M$ is stable under $g$, and therefore $g^{-1}$ sends $M$ and $N$ onto $F_M$ and $T_0N$, respectively. This proves the lemma.
\end{proof}

\subsubsection{Manifolds of bounded complexity in $G$}

The group $G$ is a simple Lie group. We fix a Euclidean structure on its Lie algebra $\g$ and endow $G$ with the corresponding left-invariant Riemannian metric. Then we make the following definition.

\begin{definition}
A \emph{submanifold chunk of complexity} $\rho^{-1}$ in $G$ is the image of a submanifold chunk in $\g$ under the exponential map.
\end{definition}

Again, we will need to know that chunks of bounded complexity are stable under two simple operations: translation by an element of $G$ and intersection.
We start by showing that we may take images of submanifold chunks under translations.

\begin{lemma}[Image of chunks under translations]
There exists a constant $C\geq 2$ depending on $G$ only such that the following holds for any $\rho\in(0,\frac{1}{2})$.\\
If $M$ is a chunk of complexity $\rho^{-1}$ in $G$, then, for all $a\in M\cap B_{\rho^C}$, $a^{-1}M$ is a chunk of complexity $\rho^{-C}$.
\end{lemma}
\begin{proof}
Again, we identify a neighborhood of the identity in $G$ with a neighborhood of $0$ in $\g$.
Write $M=f(T)$ for some linear subspace $T$ and some diffeomorphism $f$ of complexity $\rho^{-1}$.
As $f$ is invertible on a ball of radius $\rho^C$ around zero, we may define an element $t\in T$ by $t=f^{-1}(a)$.
Denote by $m_{a}$ the left multiplication by $a$ in $G$ ($m_a(x)=a*x$) and by $\tau_t$ the left translation by $t$ ($\tau_t(x)=x+t$). Noting that $\tau_t(T)=T$, we find,
$$a^{-1}M=m_{a}^{-1}\circ f (T)=m_a^{-1}\circ f\circ \tau_t (T).$$
However, it is easily seen that $m_a^{-1}\circ f\circ\tau_t$ is a diffeomorphism of complexity $\rho^{-C}$, so this proves the lemma.
\end{proof}

We now turn to intersection of transverse chunks, proving the analog of Lemma~\ref{linearchunk}, for chunks of bounded complexity of $G$. In fact, what we prove now is slightly stronger, because we also allow small translations under elements of $G$.

\begin{lemma}[Intersection of transverse chunks in $G$]\label{intersection}
There exists a constant $C\geq2$ such that the following holds for each $\rho\in(0,\frac{1}{2})$.\\
Let $M$ and $N$ be two submanifold chunks of complexity $\rho^{-1}$ in $G$, and satisfying the following:
\begin{itemize}
\item $\dim M= d-1$
\item $d(T_1N,T_1M)\geq\rho$.
\end{itemize}
Then, for all $g\in B_{\rho^C}$ and for all $a\in M\cap gN\cap B_{\rho^C}$, $a^{-1}(M\cap gN)$ is a submanifold chunk of complexity $\rho^{-C}$. Moreover, for all $x\in B_{\rho^C}$,
$$d(x,M\cap gN)\leq \rho^{-C}\cdot \max\{d(x,M),\,d(x,gN)\}.$$
\end{lemma}
\begin{proof}
Again, we identify a neighborhood of the identity in $G$ with a neighborhood of $0$ in $\g$.
From the previous lemma, for $g\in B_{\rho^C}$ and $a\in M\cap gN\cap B_{\rho^C}$, both $a^{-1}M$ and $a^{-1}gN$ are chunks of complexity $\rho^{-C}$. Moreover, for $a$ and $g$ in $B_{\rho^C}$, we have $d(T_0(a^{-1}M),T_0M)\leq \rho^C$ and $d(T_0(a^{-1}gN),T_0N)\leq \rho^C$. This implies,
$$d(T_0(a^{-1}N),T_0(a^{-1}gM))\geq \rho-2\rho^C\geq\rho^C,$$
provided $C\geq 2$, which we may of course ensure.
Thus, Lemma~\ref{linearchunk} applies, and we may find a diffeomorphism $\theta$ of complexity $\rho^{-C}$ sending $a^{-1}M$ and $a^{-1}gN$ to linear subspaces forming an angle of $\frac{\pi}{2}$. This proves the first part of the lemma. The second part is clearly true when $a$ is the identity  and $M$ and $N$ are linear subspaces of $\g$ forming an angle of $\frac{\pi}{2}$, and we can always reduce to that case, using the above diffeomorphism $\theta$. So we are done.
\end{proof}

If $U$ is a neighborhood of the identity in $G$, we make the following definition.

\begin{definition}[Submanifold of bounded complexity]
A \emph{submanifold $M$ of complexity $\rho^{-1}$ in $U$} is a submanifold of $G$ such that for each point $x$ in $M\cap U$, $x^{-1}M$ is included in a submanifold chunk of complexity $\rho^{-1}$.
\end{definition}

(Note that submanifolds of bounded complexity in $U$ are not necessarily closed subsets.)

\begin{example}
A zero-dimensional submanifold of complexity $\rho^{-1}$ in $U$ is a union of points that are at distance at least $\rho$ from each other. For volume reasons, if $U$ is bounded, the cardinality of a zero-dimensional submanifold of complexity $\rho^{-1}$ in $U$ is $O(\rho^{-d})$.
\end{example}

\subsubsection{Larsen-Pink type estimates in codimension 1}

With the above lemmas at hand, we may now derive the metric Larsen-Pink type estimates that will be used in the proof of the product theorem.

\begin{proposition}[Larsen-Pink type inequality]\label{lp}
Let $G$ be a simple Lie group of dimension $d$.
There exists a neighborhood $U$ of the identity and a constant $C\geq 0$ depending only on $G$ such that the following holds for any $\epsilon>0$ and any $\delta>0$ sufficiently small.\\
Let $A$ be a subset of $U$ that is not included in a neighborhood of size $\delta^\epsilon$ of any closed connected subgroup and suppose $A$ satisfies
$$N(AAA,\delta)\leq\delta^{-\epsilon}N(A,\delta).$$
Let $M$ be a submanifold of positive codimension and complexity at most $\delta^{-\epsilon}$ in $U$. Then,
$$N(A\cap M,\delta) \leq \delta^{-C\epsilon}N(A,\delta)^{1-\frac{1}{d}}.$$
\end{proposition}

The proposition will follow from repeated application of the following lemma.

\begin{lemma}\label{inter}
Let $G$ be a simple Lie group of dimension $d$.
There exists a neighborhood $U$ of the identity in $G$ and a constant $C$ depending on $G$ only such that the following holds for any $\epsilon>0$ and any $\delta>0$ small enough (depending on $\epsilon$).\\
Let $A$ be a symmetric subset of $U$ such that for any unit vector $v$ in $\g$ and for any hyperplane $W< \g$, there exists $a$ in $A$ such that $d((\Ad a)v,W)\geq \delta^\epsilon$.\\
Suppose $M$ and $N$ are two submanifolds of complexity $\delta^{-\epsilon}$ in $U$ and with respective dimensions $d-1$ and $n$. Then there exists a submanifold $P$ of dimension $n-1$ and complexity $\delta^{-C\epsilon}$ in $U$ such that
$$N(A^{(\delta^{1-C\epsilon})}\cap P,\delta^{1-C\epsilon})\cdot N(A^6,\delta) \geq \delta^{C\epsilon} N(A\cap M,\delta)\cdot N(A\cap N,\delta).$$
\end{lemma}

Recall that for any set $S$, we denote by $S^{(\rho)}$ the $\rho$-neighborhood of $S$, which is not to be confused with the $k$ product set, denoted $A^k$.

\begin{proof}
Let $C_0$ be the constant given by Lemma~\ref{intersection}.
As $U$ can be covered by $O(\delta^{-dC_0\epsilon})$ balls of radius $\frac{\delta^{C_0\epsilon}}{8}$ we may find $a$ in $A\cap M$ such that
$$N(A\cap M\cap B(a,\frac{\delta^{C_0\epsilon}}{4}),\delta) \gg \delta^{dC_0\epsilon} N(A\cap M, \delta).$$
Similarly, we may find $b$ in $A\cap N$ such that
$$N(A\cap N\cap B(b,\frac{\delta^{C_0\epsilon}}{4}),\delta) \gg \delta^{dC_0\epsilon} N(A\cap N, \delta).$$
Let $M'=a^{-1}M$. By the assumption on $A$, there exists $a_1$ in $A$ such that $d((\Ad a_1)T_1b^{-1}N,T_1M') \geq \delta^\epsilon$, and we let $N'= a_1 b^{-1}N a_1^{-1}$, so that $M'$ and $N'$ are submanifold chunks of complexity $\delta^{-C\epsilon}$ satisfying
\begin{itemize}
\item $N(a^{-1}A\cap M'\cap B(1,\frac{\delta^{C_0\epsilon}}{4}),\delta) \geq \delta^{O(\epsilon)}N(A\cap M,\delta)$
\item $N(a_1b^{-1}Aa_1^{-1}\cap N'\cap B(1,\frac{\delta^{C_0\epsilon}}{4}),\delta) \geq \delta^{O(\epsilon)}N(A\cap N,\delta)$
\item $\dim M'= d-1$ and $\dim N'=n$
\item $d(T_1N', T_1M') \geq \delta^{\epsilon}$.
\end{itemize}
Consider the map
$$
\begin{array}{lccc}
\psi: & a^{-1}A\cap M' \times a_1b^{-1}Aa_1^{-1}\cap N' & \longrightarrow & A^6\\
& (x,y) & \longmapsto & xy^{-1}
\end{array}
$$
Let $X$ and $Y$ be maximal $\delta$-separated subsets of $a^{-1}A\cap M'\cap B(1,\frac{\delta^{C_0\epsilon}}{4})$ and $a_1b^{-1}Aa_1^{-1}\cap N'\cap B(1,\frac{\delta^{C_0\epsilon}}{4})$ respectively, so that
$$\card X\times Y \geq \delta^{O(\epsilon)}N(A\cap M,\delta)N(A\cap N,\delta).$$
Take a cover $\mathcal{B}$ of $A^6$ by balls of radius $\delta$ such that
$$\card \mathcal{B} = N(A^6,\delta).$$
Counting points of $X\times Y$ according to their image under $\psi$, we find
\begin{align*}
\card X\times Y & = \sum_{B\in\mathcal{B}} \card\{ (x,y)\in X\times Y \,|\, xy^{-1}\in B\}\\
& \leq N(A^6,\delta)\cdot \max_{B\in\mathcal{B}} \card\{ (x,y)\in X\times Y \,|\, xy^{-1}\in B\}
\end{align*}
so that for some $g$ in the image of $\psi$, with $g\in B(1,\delta^{C_0\epsilon})$,
$$\card X\times Y \leq N(A^6,\delta)\cdot \card\{ (x,y)\in X\times Y \,|\, d(xy^{-1},g)\leq \delta\}.$$
However, as $X$ and $Y$ are $\delta$-separated,
\begin{align*}
\card\{ (x,y)\in X\times Y \,|\, d(xy^{-1},g)\leq \delta\} & \ll \card\{ x\in X \,|\, x\in (gN')^{(\delta)}\cap M'\}\\
& \leq N(a^{-1}A\cap M'\cap (gN')^{(\delta)},\delta)
\end{align*}
Now, from Lemma~\ref{intersection}, the intersection $M'\cap (gN')^{(\delta)}$ is included in the $\delta^{1-C_0\epsilon}$-neighborhood of a submanifold chunk $P_0$ of complexity $\delta^{-C_0\epsilon}$ for which we therefore have
\begin{align*}
N(a^{-1}A\cap P_0^{(\delta^{1-O(\epsilon)})},\delta) N(A^6,\delta) & \leq \delta^{-O(\epsilon)} N(A\cap M,\delta)\cdot N(A\cap N, \delta).
\end{align*}
To conclude, it suffices to take $P=aP_0$ and to note that for any set $S$,
\begin{equation}\label{presque}
N(S,\delta)\leq \delta^{-O(\epsilon)}N(S,\delta^{1-\epsilon}).
\end{equation}
\end{proof}

\begin{proof}[Proof of Proposition~\ref{lp}]
We may assume without loss of generality that $G$ has trivial center.
Any submanifold of positive codimension and complexity at most $\delta^{-\epsilon}$ is included in a submanifold of dimension $d-1$ and complexity at most $\delta^{-\epsilon}$, so it suffices to prove the proposition in the case $M$ has codimension 1.\\
The set $A$ is $\delta^\epsilon$-away from subgroups, so that from Proposition~\ref{escape} and Remark~\ref{irred} applied to the adjoint representation, there exists a constant $C$ such that the set $A'= (A\cup A^{-1}\cup \{1\})^C$ satisfies the hypotheses of Lemma~\ref{inter} (with $\epsilon$ replaced by $C\epsilon$). From Ruzsa's inequality (see Tao~\cite[Theorem~6.8]{taoestimates}), we have
$$N(A',\delta) \leq \delta^{-O(\epsilon)} N(A,\delta)$$
and therefore, it suffices to prove the proposition for the set $A'$. In other terms, we may assume that $A$ satisfies the hypothesis of Lemma~\ref{inter}.
Let $M$ be a $(d-1)$-dimensional submanifold of $U$ of complexity at most $\delta^{-\epsilon}$. We apply Lemma~\ref{inter} to the pair of manifolds $(M,M)$, thereby obtaining a submanifold $M_2$ of codimension $2$ and complexity $\delta^{-O(\epsilon)}$ such that
$$N(A^{(\delta^{1-O(\epsilon)})}\cap M_2,\delta^{1-O(\epsilon)})
\cdot N(A^6,\delta)
\geq
\delta^{O(\epsilon)}N(A\cap M,\delta) N(A\cap M,\delta).$$
Now, apply Lemma~\ref{inter} again, to the set $A^{(\delta^{1-O(\epsilon)})}$ and to the pair of manifolds $(M,M_2)$, at scale $\delta^{1-O(\epsilon)}$. This yields a manifold $M_3$ of codimension $3$ and complexity $\delta^{-O(\epsilon)}$ such that
\begin{IEEEeqnarray*}{rCl}
\IEEEeqnarraymulticol{3}{l}{
N(A^{(\delta^{1-O(\epsilon)})}\cap M_3,\delta^{1-O(\epsilon)})
\cdot N((A^{(\delta^{1-O(\epsilon)})})^6,\delta)}\\
\qquad\qquad\qquad\qquad\qquad & \geq &
\delta^{O(\epsilon)}N(A^{(\delta^{1-O(\epsilon)})}\cap M,\delta^{1-O(\epsilon)}) N(A^{(\delta^{1-O(\epsilon)})}\cap M_2,\delta).
\end{IEEEeqnarray*}
Then repeat this procedure $d-1$ times to obtain at the end a zero-dimensional submanifold $M_d$ of complexity $\delta^{-O(\epsilon)}$ such that
\begin{IEEEeqnarray*}{rCl}
\IEEEeqnarraymulticol{3}{l}{
N(A^{(\delta^{1-O(\epsilon)})}\cap M_d,\delta^{1-O(\epsilon)})
\cdot N((A^{(\delta^{1-O(\epsilon)})})^6,\delta)}\\
\qquad\qquad\qquad\qquad\qquad & \geq &
\delta^{O(\epsilon)}N(A^{(\delta^{1-O(\epsilon)})}\cap M,\delta^{1-O(\epsilon)}) N(A^{(\delta^{1-O(\epsilon)})}\cap M_{d-1},\delta).
\end{IEEEeqnarray*}
Taking the product of all the obtained inequalities and making the obvious simplifications, we get
$$N(A^{(\delta^{1-O(\epsilon)})}\cap M_d,\delta^{1-O(\epsilon)})N((A^{(\delta^{1-O(\epsilon)})})^6,\delta)^{d-1} \geq \delta^{O(\epsilon)}
N(A\cap M,\delta^{1-O(\epsilon)})^d.$$
However, $M_d$ being a zero-dimensional submanifold of complexity $\delta^{-O(\epsilon)}$, it is a finite set of cardinality at most $\delta^{-O(\epsilon)}$ and therefore,
$$N((A^{(\delta^{1-O(\epsilon)})})^6,\delta)^{d-1} \geq \delta^{O(\epsilon)}
N(A\cap M,\delta^{1-O(\epsilon)})^d,$$
from which one readily concludes, using once more Rusza's inequality and the trivial inequality (\ref{presque}), that
$$N(A\cap M,\delta) \leq \delta^{-O(\epsilon)}N(A^6,\delta)^{1-\frac{1}{d}} \leq \delta^{-O(\epsilon)} N(A,\delta)^{1-\frac{1}{d}}.$$
\end{proof}

\section{Proof of the product theorem}\label{expansion}

\subsection{Rich torus}

The starting point of the proof of the product theorem is the following: from a small tripling set $A$, find a maximal torus whose $\delta$-neighborhood contains many elements of $A$.
For that, we first show that some product set of $A$ contains a very regular element.
Recall that an element $g$ in $G$ is called \emph{regular} (or \emph{regular semisimple}) if the multiplicity of the eigenvalue $1$ in the matrix representation $\Ad g$ is minimal. If $g$ is not regular, we will call it \emph{singular}.
We denote by $\mathcal{S}$ the set of singular elements of $G$, i.e.
$$\mathcal{S}=\{x\in G\,|\, \mbox{the multiplicity of}\ 1\ \mbox{as an eigenvalue of}\ \Ad x\ \mbox{is not minimal}\}.$$

\begin{lemma}\label{aregular}
Let $G$ be a simple Lie group, and denote by $\mathcal{S}$ the set of singular elements in $G$. There exists a neighborhood $U$ of the identity in $G$ and a constant $C$ such that the following holds.\\
If $A\subset U$ is $\rho$-away from subgroups, then there exists an element $a\in A^C$ such that $d(a,\mathcal{S})\geq \rho^C$.
\end{lemma}
\begin{proof}
We may assume without loss of generality that $G$ has trivial center, and view it as a subvariety of $\mathcal{M}_n(\C)$, the $n$ by $n$ matrices over $\mathbb{C}$. Then, $U$ is chosen as in Proposition~\ref{almoststab}.
The set $\mathcal{S}$ is a proper algebraic subvariety of $G$, so we may choose a polynomial $P$ that vanishes on $\mathcal{S}$, but not on $U$. Let $V<\C[{x_{ij}}{1\leq i,j\leq n}]$ be the finite-dimensional subrepresentation of $G$ generated by $P$, and $W=\{Q\in V\,|\, Q(1)=0\}$. Taking $c=\sup_{g\in U} |P(g)|$, we may apply Proposition~\ref{escape} and find $a\in A^C$ such that $d(a\cdot P,W)\geq\rho^C$, i.e. $|P(a)|\geq\rho^C$. As $P$ is a Lipschitz function on $U$, this certainly implies that $d(a,\mathcal{S})\geq\rho^C$ (again $C$ may have increased from one line to the other).
\end{proof}

From now on, we will restrict attention to a bounded neighborhood $U$ of the identity in which Proposition~\ref{lp} and Lemma~\ref{aregular} hold.

\begin{lemma}
There exists a constant $C\geq 0$ depending only on $G$ such that the following holds for any $\rho\in(0,\frac{1}{2})$.
Let $a$ in $U$ be an element such that $d(a,\mathcal{S})\geq\rho$. Then, the conjugacy class $C_a$ of $a$ is a submanifold of complexity at most $\rho^{-C}$ in $U$.
\end{lemma}
\begin{proof}
For each $x$ in $C_a$, we have $C_x=C_a$, and $d(x,\mathcal{S})\gg \rho$, so it suffices to show that $a^{-1}C_a$ is a manifold chunk of complexity $\rho^{-C}$. Denote by $T$ the maximal torus of $G$ containing $a$, by $\mathfrak{t}$ its Lie algebra, and decompose the Lie algebra $\g_{\C}$ into root spaces:
$$\g_\C = \left(\bigoplus_{\alpha\in \Delta} \g_\alpha\right)\oplus \mathfrak{t}_{\C}.$$
In a neighborhood of the identity, any element $g\in G$ can be written $g=e^Xe^t$ for some $X\in\g':=\mathfrak{g}\cap\bigoplus_{\alpha\in\Delta}\mathfrak{g}_\alpha$ and $t\in\mathfrak{t}$.
Therefore, in a neighborhood of the identity, any element $b\in a^{-1}C_a$ can be written $a^{-1}e^Xae^{-X}=e^{(\Ad a^{-1})X}e^{-X}$, for some $X\in\g'$.
Once more, we identify a neighborhood of the identity in $G$ with a neighborhood of $0$ in $\g$. Let
$$
\begin{array}{cccc}
\varphi: & \g & \rightarrow & \g\simeq G\\
& (X+t) & \mapsto & e^te^{(\Ad a^{-1})X}e^{-X}.
\end{array}
$$
The differential of $\varphi$ at $0$ is
$$
\begin{array}{cccc}
\varphi'(0): & \g & \rightarrow & \g\\
& (X+t) & \mapsto & t+(\Ad a^{-1}-1)X.
\end{array}
$$
An eigenvalue $\lambda$ of $\varphi'(0)$ in $\C$ is either $1$ or $\chi_\alpha(a^{-1})-1$ where $\chi_\alpha$ is the character of $T$ corresponding to the root $\alpha$; as $d(a,\mathcal{S})\geq \rho$, we must have $|\lambda|\geq\rho$. Since the operator norm of $\varphi'(0)$ is bounded by a constant depending on $U$ only, this also implies $\|\varphi'(0)^{-1}\|\leq C\rho^{-1}$, for some $C$ depending only on $U$.
Therefore, $\varphi'(0)$ is $\rho^{-C}$-bi-Lipschitz. As of course, $\varphi(0)=0$ and $\varphi'$ is $C$-Lipschitz for some constant $C$ depending on $U$ only, $\varphi$ is a diffeomorphism of complexity $\rho^{-C}$. But $\varphi(\g')=C_a$ in a neighborhood of the identity, so the lemma is proved.
\end{proof}

Combining the above lemma and the Larsen-Pink type inequality, we finally obtain the rich torus we were looking for:

\begin{corollary}\label{richtorus}
Let $G$ be a simple Lie group.
There exists a neighbhorhood $U$ of the identity in $G$ and a constant $C\geq 0$ depending only on $G$ such that for $\delta>0$ small enough, the following holds.\\
Let $A$ be a symmetric subset of $U$ that is not included in a neighborhood of size $\delta^\epsilon$ of a closed subgroup and satisfying
$$N(AAA,\delta)\leq\delta^{-\epsilon}N(A,\delta).$$
Then, there exists a maximal torus $T$ of $G$ such that
$$N(A^{-1}A\cap T^{(\delta^{1-C\epsilon})},\delta) \geq \delta^{C\epsilon}N(A,\delta)^{\frac{1}{d}}.$$
\end{corollary}
\begin{proof}
From Lemma~\ref{aregular}, there exists an element $a$ in a product set of $A$ such that $d(a,\mathcal{S})\geq\delta^{C\epsilon}$. We let $A$ act on $C_a$ by conjugation.
 From the previous lemma, $C_a$ is a submanifold of complexity $\delta^{-C\epsilon}$ in $U$, so from the Larsen-Pink type inequality (Proposition~\ref{lp}),
$$N(A^C\cap C_a,\delta)\leq \delta^{-C\epsilon} N(A,\delta)^{1-\frac{1}{d}}.$$
Therefore, by Dirichlet's box-principle, there exists $g\in C_a$ such that
$$N(\{x\in A\,|\, d(xax^{-1},g)\leq \delta\},\delta)\geq \delta^{C\epsilon}N(A,\delta)^{\frac{1}{d}}.$$
Choosing $x_0\in A$ such that $d(x_0ax_0^{-1},g)\leq \delta$, we find
\begin{align*}
\delta^{C\epsilon}N(A,\delta)^{\frac{1}{d}} & \leq N(\{x\in A\,|\, d(x_0^{-1}xa(x_0^{-1}x)^{-1},x_0^{-1}gx_0)\leq \delta\},\delta)\\
& \leq N(\{x\in A\,|\, d(x_0^{-1}xax^{-1}x_0,a)\leq 2\delta\},\delta)\\
& \leq N(\{x\in A^{-1}A\,|\, d(xax^{-1},a)\leq 2\delta\},\delta).
\end{align*}
To conclude, it will suffice to show that for $x$ in $U$,
$$d(xax^{-1},a)\leq 2\delta \Longrightarrow d(x,T)\leq \delta^{1-C\epsilon}.$$
For this, write $x=e^X$, and decompose $X$ onto the root-spaces of $\g_\C$:
$$X= t + \sum_{\alpha\in\Delta} X_\alpha$$
with, $t\in\mathfrak{t}$ and, for each $\alpha$, $X_\alpha\in\g_\alpha$.
As $\exp$ is a diffeomorphism on a neighborhood of the identity, we get from $d(xax^{-1},a)\leq 2\delta$ that  $d(X,(\Ad a)X)\leq C\delta$ for some constant $C$ depending only on $G$. Now, 
$$(\Ad a)X-X=\sum_{\alpha\in\Delta}(\chi_\alpha(a)-1)X_\alpha,$$
and, as $d(a,\mathcal{S})\geq \delta^{C\epsilon}$, we have for each $\alpha\in\Delta$, $|\chi_\alpha(a)-1|\geq\delta^{C\epsilon}$. Thus, we get, for each $\alpha$, $\|X_\alpha\|\leq \delta^{1-C\epsilon}$, i.e.
$$d(X,\mathfrak{t})\leq \delta^{1-C\epsilon}.$$
Going back to $G$, this translates to
$$d(x,T) \leq \delta^{1-C\epsilon},$$
which is exactly what we wanted to show.
\end{proof}

\subsection{From a rich torus to a small segment}

The fundamental growth statement we use in our proof of the product theorem is the following lemma of Bourgain and Gamburd \cite[Corollary 8]{bourgaingamburdsud}.\\
Denote $\Delta \subset \Mat_{d\times d}(\C)$ the set of diagonal matrices. If $A$ is a subset of an additive group, and $s$ a positive integer, we denote by $sA$ the $s$-fold sumset $A+\dots+A$.

\begin{lemma}\label{bgproduct}
Given $\sigma>0$ and $d$ a positive integer, there exist $\alpha\geq 0$, $\beta >0$, and a positive integer $s$ such that, for $\delta>0$ sufficiently small, the following holds.\\
Assume $A\subset \Mat_{d\times d}(\C)$ satisfies
\begin{enumerate}
\item $A\subset B(0,2)$
\item $N(A,\delta)>\delta^{-\sigma}$
\item $d(x,\Delta)<\delta$ for $x\in A$.
\end{enumerate}
Then there is $\eta\in\Delta$, $\|\eta\|=1$ such that
$$[0,\delta^\alpha]\eta\subset sA^s-sA^s+B(0,\delta^{\alpha+\beta}).$$
\end{lemma}

\begin{remark}\label{uniform}
In addition, it follows from the proof of that result that when $d$ is fixed and $\sigma$ remains bounded away from zero, the corresponding constants $\alpha$ and $s$ remain bounded, while $\beta$ remains bounded away from zero.
\end{remark}

The idea is to apply that lemma in the adjoint representation of $G$ on $\g_\C$ to a rich torus as constructed above. This will yield inside a product set of $A$ some small one-dimensional structure from which we will be able to derive the desired growth statement. In what follows, $U$ is a neighborhood of the identity in $G$ in which all the above results hold: Larsen-Pink type inequalities, rich torus, ...etc. 

The following lemma will be the key step to prove Proposition~\ref{minustau}.

\begin{lemma}\label{abc}
Given $\sigma\in(0,d)$, there exist $C=C(\sigma, G)$, $\epsilon_0>0$, $\alpha\geq 0$ and $\beta>0$ such that for $\delta>0$ sufficiently small, for $\gamma\geq\alpha+\beta$, the following holds.\\
Let $A\subset U$ be a symmetric set that is $\delta^{\epsilon_0}$-away from subgroups, and such that $N(A,\delta)\geq\delta^{-\sigma}$ and
$$N(AAA,\delta)\leq\delta^{-\epsilon_0}N(A,\delta).$$
If $A$ contains an element whose distance from the identity is $\delta^\gamma$, then there exists a unit element $\xi\in \g$ such that the segment
$$\{\exp(t\xi)\,;\, t\in[0,\delta^{\alpha+\gamma}]\}$$
is included in a neighborhood of size $\delta^{\alpha+\beta+\gamma}$ of $A^C$.
\end{lemma}

The proof consists of applying Lemma~\ref{bgproduct} to a rich torus as given by Corollary~\ref{richtorus}, in the adjoint representation. However, in order to prevent the sum operation from producing an element too far away from a product set $A^C$, we must act by conjugation on an element whose distance to the identity is less than $\delta^{\alpha+\beta}$ -- whence the condition $\gamma\geq\alpha+\beta$. 

\begin{proof}
Suppose $A\subset U$ is a set satisfying the hypotheses of the lemma.
From Corollary~\ref{richtorus}, there exists a maximal torus $T$ with
$$N(A^{-1}A\cap T^{(\delta^{1-C\epsilon_0})},\delta)\geq \delta^{-\frac{\sigma}{d}+C\epsilon_0}.$$
Let $B$ be the image of $A^{-1}A\cap T^{(\delta^{1-C\epsilon_0})}$ under the adjoint representation $\Ad$ of $G$ on $\g_{\C}$. The representation is bounded, so that provided $U$ has been chosen small enough, we have
$$B\subset B(0,2).$$
From the decomposition of $\g_{\C}$ into weight-spaces, $\Ad T$ can be viewed as a subset of the diagonal matrices of size $d$ and so for each $b\in B$,
$$d(b,\Delta)\leq\delta^{1-C\epsilon_0}:=\delta_1.$$
Finally,  on a neighborhood of the identity, the adjoint map $g\mapsto \Ad g$ is a diffeomorphism, so that
$$N(B,\delta_1)\gg N(A\cap T^{(\delta_1)},\delta_1)\gg \left(\frac{\delta_1}{\delta}\right)^d N(A\cap T^{(\delta_1)},\delta)\geq \delta_1^{-\frac{\sigma}{d}+C\epsilon_0}.$$
So we may apply Lemma~\ref{bgproduct} to $B$: there exists an $\eta\in\Delta$, $\|\eta\|=1$ such that
$$[0,\delta^\alpha]\eta\subset sB^s-sB^s+B(0,\delta^{\alpha+\beta}).$$
Now let $X$ be a unit element of $\g$, the Lie algebra of $G$.
For $t\in[0,\delta^\alpha]$, write
$$t\eta=\Ad x_1 \pm \Ad x_2 \pm \dots \pm \Ad x_{2s} + O(\delta^{\alpha+\beta}),$$
where each $x_i$ is an element of $A^s$.
If $u>0$ is another parameter, we have:
\begin{align*}
\exp(t\eta(uX)) & = \exp[(\Ad x_1 \pm \Ad x_2 \pm \dots \pm \Ad x_{2s} + O(\delta^{\alpha+\beta})(uX)]\\
& = e^{(\Ad x_1)(uX)}e^{\pm(\Ad x_2)(uX)}\dots e^{\pm(\Ad x_{2s})(uX)} e^{O(u\delta^{\alpha+\beta}+u^2)}\\
& = x_1 e^{uX} x_1^{-1} x_2 e^{\pm uX} x_2^{-1}\dots x_{2s}e^{\pm uX} x_{2s}^{-1} e^{O(u\delta^{\alpha+\beta}+u^2)}.
\end{align*}
Write $u=\delta^{\gamma}$ with $\gamma\geq\alpha+\beta$, and choose an element $g=e^{uX}$ in $A$. We find, for each $t\in [0,\delta^\alpha]$,
$$d(e^{ut\eta(X)}, A^C)=O(u\delta^{\alpha+\beta}).$$ 
If $\|\eta(X)\|\geq \delta^\epsilon$, this shows that some product set of $A$ (with bounded exponent) contains a segment of length $\delta^{\alpha+\gamma}$ in its $\delta^{\alpha+\beta+\gamma}$-neighborhood (adjusting the value of $\beta$ by some $\epsilon$).

On the other hand from Proposition~\ref{escape} applied in the adjoint representation, with vector $X$ and subspace $\ker\eta$, we may always find an element $a\in A^C$ such that
$$\|\eta(\Ad a)X)\|\geq \delta^{C\epsilon_0}.$$
As the element $e^{u(\Ad a)X}=ag a^{-1}$ is in $A^{3C}$, we may replace $X$ by $(\Ad a)X$ in the above computation, so that the element $\xi=\eta(\Ad a)X$ satisfies the conclusion of the lemma.
\end{proof}

To prove the product theorem, we will now make use of the scale invariance assumption on $A$: for all $\rho\geq\delta$,
$$N(A,\rho)\geq\delta^\epsilon\rho^{-\sigma}.$$
Using this property, we may improve the previous lemma, this is the content of the next proposition.

\begin{proposition}\label{minustau}
Given $\sigma\in(0,d)$, there exist $C\geq 0$ and $\tau,\epsilon_1>0$ such that for $\delta>0$ sufficiently small, the following holds.\\
Assume $A$ is a symmetric set in $U$ satisfying:
\begin{enumerate}
\item $N(A,\delta) \leq \delta^{-\sigma-\epsilon_1}$
\item $\forall \rho\geq\delta$, $N(A,\rho)\geq\delta^{\epsilon_1}\rho^{-\sigma}$
\item $A$ is $\delta^{\epsilon_1}$-away from subgroups
\item $N(AAA,\delta)\leq \delta^{-\epsilon_1}N(A,\delta)$.
\end{enumerate}
Then, there exists a segment of length $\delta^{1-\tau}$,
$$\{\exp(t\xi)\,;\, t\in[0,\delta^{1-\tau}]\}$$
that is included in a $\delta$-neighborhood of $A^C$.
\end{proposition}

Note that these four conditions become more restrictive when $\epsilon_1$ becomes smaller, and that the aim of this paper is to show that for $\epsilon_1$ small enough, these conditions are incompatible.

\begin{proof}
First note that, provided $\epsilon_1>0$ is sufficiently small, we have
$$N(A,\delta^{1/4}) \geq \delta^{-\frac{\sigma}{4}+\epsilon_1} > \delta^{-\frac{\sigma}{5}} \geq N(A, \delta^{\frac{\sigma}{5d}})$$
so that there exist $x$ and $y$ in $A$ with
$$\delta^{\frac{1}{4}} \leq d(x,y) \leq 2\delta^{\frac{\sigma}{5d}}.$$
In other terms there is an element $a_0\in AA^{-1}$ whose distance to the identity is $\delta^{\kappa}$, with 
$$\frac{1}{4} \geq \kappa \geq \frac{\sigma}{6d}.$$
From $a_0$, we define inductively a sequence of elements $a_k$, in the following way:\\
write $a_k=e^{X_k}$ and apply Proposition~\ref{escape} in the adjoint representation, with vector $X_k$ and subspace $\ker\Ad a_0 -1$, to get an element $x_k$ in some $A^C$ such that
$$d((\Ad x_k)X_k,\ker\Ad a_0-1) =\delta^{O(\epsilon_1)},$$
and let $a_{k+1}=[a_0,x_ka_kx_k^{-1}]$, so that
$$d(a_{k+1},1)=\delta^{O(\epsilon_1)}d(a_0,1)d(a_k^{x_k},1).$$
This ensures that for bounded $k$'s,
$$d(a_k,1)=\delta^{k\kappa+O(\epsilon_1)}.$$
In particular, for some $k\leq\frac{5d}{\sigma}$ we get an element $a_k$ in a product set of $A$ with $d(a_k,1)=\delta^{\gamma_0}$ and
$$\frac{1}{2}\leq \gamma_0 \leq \frac{3}{4}.$$
Now let $\alpha$ and $\beta$ be the parameters given by the previous lemma and define $\gamma=\frac{\gamma_0(\alpha+\beta)}{1-\gamma_0}$, so that
$$\frac{\gamma}{\alpha+\beta+\gamma}=\gamma_0,$$
and
$$\alpha+\beta \leq \gamma \leq 3(\alpha+\beta).$$
Choosing $\epsilon_1>0$ smaller than $\frac{\epsilon_0}{3(\alpha+\beta+\gamma)}$ ensures that the set $A$ viewed at scale $\delta_1=\delta^{\frac{1}{\alpha+\beta+\gamma}}$ satisfies the hypotheses of Lemma~\ref{abc}.
Indeed, one readily checks that the first three conditions are satisfied. For the fourth, note that by chosing a ball $B(x_1,\delta_1)$ of radius $\delta_1$ such that $N(A\cap B(x_1,\delta_1),\delta) \geq\frac{N(A,\delta)}{N(A,\delta_1)}$ and by translating it along points of a $2\delta_1$-separated subset of $AAA$, we find
$$N(AAAA,\delta) \gg \frac{N(A,\delta)}{N(A,\delta_1)}N(AAA,\delta_1),$$
so that
$$N(AAA,\delta_1) \ll \frac{N(AAAA,\delta)}{N(A,\delta)} N(A,\delta_1) \leq \delta^{-3\epsilon_1} N(A,\delta_1).$$
The element $a_k\in A^C$ constructed above satisfies $d(a_k,1)=\delta_1^{\gamma}$, and $\gamma\geq \alpha+\beta$. So we may apply Lemma~\ref{abc} to $A$ at scale $\delta_1$. This shows that a product set $A^C$ of $A$ contains a neighborhood of size $\delta$ of a segment of length $\delta^{1-\frac{\beta}{\alpha+\beta+\gamma}}\geq \delta^{1-\tau}$, where $\tau=\frac{\beta}{4(\alpha+\beta)}$.
\end{proof}

\subsection{From a small segment to the whole ambient group}

From the one-dimensional structure constructed in the previous subsection, we now recover the whole ambient group $G$ in some product set of $A$, hence reaching a contradiction.

\begin{lemma}\label{smallball}
Given $\sigma\in (0,d)$, there exists a constant $C\geq 0$ and $\tau,\epsilon_2>0$ such that for $\epsilon\in(0,\epsilon_2)$, for $\delta>0$ sufficiently small, the following holds.\\
Assume $A$ is a symmetric set in $U$ satisfying
\begin{enumerate}
\item $N(A,\delta)\leq\delta^{-\sigma-\epsilon}$
\item $\forall\rho\geq\delta$, $N(A,\rho)\geq\delta^{\epsilon}\rho^{-\sigma}$
\item $A$ is $\delta^\epsilon$-away from subgroups
\item $N(AAA,\delta)\leq \delta^{-\epsilon}N(A,\delta).$
\end{enumerate}
Then,
$$N(A^C\cap B_{\delta^{1-\tau}},\delta) \geq \delta^{-d\tau+O(\epsilon)},$$
where $B_{\rho}$ is the ball of radius $\rho$ centered at the identity in $G$.
\end{lemma}
\begin{proof}
Let $\tau$ be the parameter given by Proposition~\ref{minustau}.
Under the assumptions of the lemma, we have, for some unit vector $X\in \g$, for all $t\in [0,\delta^{1-\tau}]$,
$$d(e^{tX},A^C)\leq \delta.$$
From iterated application of Proposition~\ref{escape} in the adjoint representation, there exist elements $a_i$, $1\leq i\leq d$ in a product set of $A$ such that for each $i\geq 2$,
\begin{equation}\label{ajx}
d((\Ad a_i)X,\bigoplus_{j\leq i-1}\R (\Ad a_j)X) \geq \delta^{O(\epsilon)}.
\end{equation}
As $e^{(\Ad a)X}=a e^X a^{-1}$ we also have, for each $i$, for $t_i\in[0,\delta^{1-\tau}]$,
$$d(e^{t_i(\Ad a_i)X},A^C)\leq \delta,$$
and therefore, denoting $X_i=(\Ad a_i)X$,
$$d(e^{t_1X_1}e^{t_2X_2}\dots e^{t_dX_d}, A^C) = O(\delta).$$
The differential at zero of the map
$$\varphi:(t_i)\mapsto e^{t_1X_1}e^{t_2X_2}\dots e^{t_dX_d}$$
is
$$\varphi'(0):(t_i)\mapsto t_1X_1+t_2X_2+\dots+t_dX_d,$$
and, by (\ref{ajx}) and Lemma~\ref{bilipschitz}, $\varphi'(0)$ is $\delta^{-C\epsilon}$-bi-Lipschitz. The quantitative Local Inverse Theorem thus implies that $\varphi$ is $\delta^{-C\epsilon}$-bi-Lipschitz on a neighborhood of size $\delta^{C\epsilon}$ of $0$. In particular,
$$N(\varphi([0,\delta^{1-\tau+C\epsilon}]^d),\delta)\geq \delta^{-d\tau+O(\epsilon)},$$
so that
$$N(A^C\cap B(1,\delta^{1-\tau}),\delta) \geq \delta^{-d\tau+O(\epsilon)}.$$
\end{proof}

We are now ready to prove the Product Theorem, which we recall, for convenience of the reader:
\begin{theorem}\label{sigma}
Let $G$ be a simple Lie group of dimension $d$ and fix a small neighborhood $U$ of the identity as before. Given $\sigma\in (0,d)$, there exists $\epsilon_3=\epsilon_3(\sigma)>0$ such that, for $\delta>0$ sufficiently small, if $A$ is a set in $U$ such that,
\begin{enumerate}
\item $N(A,\delta)\leq\delta^{-\sigma-\epsilon_3}$
\item $\forall\rho\geq\delta$, $N(A,\rho)\geq\delta^{\epsilon_3}\rho^{-\sigma}$
\item $A$ is $\delta^{\epsilon_3}$-away from subgroups,
\end{enumerate}
then
$$N(AAA,\delta)>\delta^{-\epsilon_3}N(A,\delta).$$
\end{theorem}
\begin{proof}
By the Plünnecke-Ruzsa inequalities \cite[Theorem~6.8]{taoestimates}, it suffices to prove the theorem in the case $A$ is symmetric.
Now assume for a contradiction that $A$ is a symmetric set satisfying the assumptions of the theorem, and that
$$N(AAA,\delta)\leq \delta^{-\epsilon}N(A,\delta).$$
Let $\epsilon_2,\tau>0$ be given by Lemma~\ref{smallball}.
If $\sigma\leq\frac{d\tau}{2}$ (say), then Lemma~\ref{smallball} immediately yields the desired contradiction. Otherwise, choose an integer $K$ such that $(1-\tau)^K\leq \frac{d-\sigma}{2d}$, and apply the lemma again, to the set $A$ viewed at each scale $\delta_k=\delta^{(1-\tau)^k}$ (choosing $\epsilon$ sufficiently small so that at each of those scales, $A$ satisfies the hypotheses of Lemma~\ref{minustau}). This shows that for some product set $A^C$ of $A$, for each $k\leq K$,
$$N(A^C\cap\delta_k,\delta_{k-1})\geq \delta_{k-1}^{-d\tau+O(\epsilon)}.$$
Therefore,
\begin{align*}
N(A^{C},\delta) & \geq N(A\cap B_{\delta_1},\delta) N(A\cap B_{\delta_2},\delta_1) \dots N(A\cap B_{\delta_K},\delta_{K-1})\\
& \geq \delta^{O(\epsilon)}(\delta\delta_1\dots\delta_{K-1})^{-d\tau}\\
& = \delta^{-d\tau(1+(1-\tau)+\dots+(1-\tau)^{K-1})+O(\epsilon)}\\
& = \delta^{-d (1-(1-\tau)^K)+O(\epsilon)}
\geq \delta^{-\frac{\sigma+d}{2}+O(\epsilon)}
\end{align*}
which, choosing $\epsilon>0$ small enough (in terms of $\sigma$ and $d$), yields a contradiction.
\end{proof}

\begin{remark}
Using Remark~\ref{uniform} and carefully examining our proof, one sees that if $\sigma$ remains pinched in an interval $[\kappa,d-\kappa]$, then the corresponding $\epsilon_3$ remains bounded away from zero. This fact will be essential for what follows.
\end{remark}

It is worth noting that one may weaken slightly the assumptions of the Product Theorem~\ref{sigma} in the following way:

\begin{theorem}\label{sigmakappa}
Let $G$ be a simple Lie group of dimension $d$. There exists a neighborhood $U$ of the identity in $G$ such that the following holds.\\
Given $\theta\in (0,d)$ and $\kappa>0$, there exists $\epsilon=\epsilon(\theta,\kappa)>0$ such that, for $\delta>0$ sufficiently small, if $A\subset U$ is a set satisfying
\begin{enumerate}
\item $N(A,\delta) \leq \delta^{-\theta-\epsilon}$,
\item for all $\rho\geq\delta$, $N(A,\rho)\geq \delta^\epsilon\rho^{-\kappa}$,\label{ro}
\item $A$ is $\delta^\epsilon$-away from subgroups,
\end{enumerate}
then
\begin{equation}\label{growth}
N(AAA,\delta)>\delta^{-\epsilon}N(A,\delta).
\end{equation}
\end{theorem}

Compared with Theorem~\ref{sigma} the nontrivial new case is when $\kappa<\theta$.
This version is the one needed for the application to the spectral gap property in compact simple Lie groups \cite{benoistsaxcesg}. The argument showing that Theorem~\ref{sigma} implies Theorem~\ref{sigmakappa} is identical, up to minor changes, to the one given by Bourgain and Gamburd in \cite{bourgaingamburdsu2} for the proof of their Proposition~3.2, but we include it for completeness.

\begin{proof}[Proof of Theorem~\ref{sigmakappa}]
Choose $\epsilon_3>0$ such that Theorem~\ref{sigma} holds for all $\sigma\in[\frac{\kappa}{2},\theta]$. Let $K$ be an integer such that $K\epsilon_3\geq \theta$ and fix a positive $\epsilon< \frac{1}{3}\epsilon_3\left(\frac{\epsilon_3}{\theta}\right)^K$. Let $A$ be a set satisfying the hypotheses of the theorem for such choice of $\epsilon$.\\
For $\sigma=\theta$, the set $A$ satisfies all hypotheses of Theorem~\ref{sigma} except \ref{ro}. Assume for a contradiction that we have
$$N(AAA,\delta)\leq \delta^{-\epsilon_3}N(A,\delta),$$
this implies that for some $\rho_1\geq\delta$,
$$N(A,\rho_1)\leq \rho_1^{-\theta}\delta^{\epsilon_3}.$$
In particular,
$$\rho_1\leq \delta^{\frac{\epsilon_3}{\theta}}.$$
We now view $A$ at scale $\rho_1$. We have $N(A,\rho_1)\leq \rho_1^{-\theta+\epsilon_3}$ and, by the choice we made on $\epsilon$,
\begin{itemize}
\item for all $\rho\geq\rho_1$, $N(A,\rho)\geq\rho^{-\kappa}\rho_1^{\epsilon_3}$
\item $A$ is $\rho_1^{\epsilon_3}$-away from subgroups.
\end{itemize}
We now iterate this procedure: assume we have defined a scale $\rho_k$ such that
\begin{itemize}
\item $\rho_k\leq \delta^{\left(\frac{\epsilon_3}{\theta}\right)^k}$
\item $N(A,\rho_k)\leq \rho_k^{-\theta+k\epsilon_3}$
\item for all $\rho\geq\rho_k$, $N(A,\rho)\geq\rho^{-\kappa}\rho_k^{\epsilon_3}$
\item $A$ is $\rho_k^{\epsilon_3}$-away from subgroups.
\end{itemize}
(Note that the second and third conditions imply that $\theta-k\epsilon_3\geq\frac{\kappa}{2}$ and $k\leq K$.) If we have
$$N(AAA,\rho_k)\leq \rho_k^{-\epsilon_3}N(A,\rho_k),$$
Theorem~\ref{sigma} yields a scale $\rho_{k+1}\geq\rho_k$ such that
$$N(A,\rho_{k+1})\leq \rho_{k+1}^{-\theta+k\epsilon_3}\rho_k^{\epsilon_3}.$$
This implies in particular
\begin{itemize}
\item $\rho_{k+1}\leq \delta^{\left(\frac{\epsilon_3}{\theta}\right)^{k+1}}$
\item $N(A,\rho_{k+1})\leq \rho_{k+1}^{-\theta+(k+1)\epsilon_3}$
\end{itemize}
and, by the choice we made on $\epsilon$,
\begin{itemize}
\item for all $\rho\geq\rho_{k+1}$, $N(A,\rho)\geq\rho^{-\kappa}\rho_{k+1}^{\epsilon_3}$
\item $A$ is $\rho_{k+1}^{\epsilon_3}$-away from subgroups.
\end{itemize}
As $K\epsilon_3\geq\theta$, this procedure must stop for some $k\leq K$. This means
$$N(AAA,\rho_k)\geq \rho_k^{-\epsilon_3}N(A,\rho_k).$$
But then,
\begin{align}
N(AAAA,\delta) & \geq N(AAA,\rho_k)\frac{N(A,\delta)}{N(A,\rho_k)}\\
& \geq \rho_k^{-\epsilon_3}N(A,\delta)\\
& \geq \delta^{-3\epsilon}N(A,\delta).
\end{align}
Using Ruzsa's inequality $N(AAA,\delta) \geq \left(\frac{N(AAAA,\delta)}{N(A,\delta)}\right)^{\frac{1}{3}}N(A,\delta)$, this yields the desired growth statement
$$N(AAA,\delta) \geq \delta^{-\epsilon}N(A,\delta).$$
\end{proof}

\bibliographystyle{plain}
\bibliography{bibliography}

\end{document}